\documentclass[12pt,twoside]{article}
\usepackage{a4wide,color}
\usepackage{amssymb}
\usepackage{amsmath}
\newtheorem{theorem}{Theorem}

\newtheorem{kg}{Theorem (\KG)}

\newtheorem{corollary}{Corollary}
\newtheorem{lemma}{Lemma}
\newtheorem{thbs}{Theorem (Schmidt)}

\def\ds{\displaystyle}
\renewcommand{\Bbb}[1]{\mathbb{#1}}
\newcommand{\I}{{\Bbb I}}         % often algebraic numbers
\newcommand{\N}{{\Bbb N}}         % natural numbers
\newcommand{\R}{{\Bbb R}}         % real numbers
\newcommand{\Rp}{{\Bbb R}^{+}}    % positive real numbers
\newcommand{\Z}{{\Bbb Z}}         % integer numbers

\def\KG{Khintchine-Groshev}

\newcommand{\cA}{{\cal A}}

\newcommand{\cH}{{\cal H}}

\newcommand{\cR}{{\cal R}}

\newcommand{\ve}{\varepsilon}

\newcommand{\diam}{\operatorname{diam}}
\newcommand{\dist}{\operatorname{dist}}

\newcommand{\vv}[1]{{\mathbf{#1}}}

\renewcommand{\le}{\leqslant}
\renewcommand{\ge}{\geqslant}
\newcommand{\nz}{\setminus\{\vv 0\}}

\newenvironment{proof}{\textsc{Proof.}}{\ \newline\hspace*{\fill}$\boxtimes$}

\def\q{\mathbf{q}}

\begin{document}

\title{Classical metric Diophantine approximation revisited: the Khintchine-Groshev theorem}

\author{Victor Beresnevich\footnote{EPSRC Advanced Research Fellow, grant EP/C54076X/1}
\\ {\small\sc York}  \and Sanju Velani\footnote{Research supported by EPSRC grants EP/E061613/1 and EP/F027028/1 }
\\ {\small\sc York}}

\date{}
%{\small\it Dedicated to ...}}

\maketitle

\begin{abstract}
Let $\cA_{n,m}(\psi)$ denote the set of $\psi$-approximable points
in $\R^{mn}$. Under the assumption that the approximating function
$\psi$ is monotonic, the classical  Khintchine-Groshev theorem
provides an elegant probabilistic criterion for the Lebesgue measure
of $\cA_{n,m}(\psi)$.  The famous Duffin-Schaeffer counterexample
shows that the monotonicity assumption on $\psi$ is absolutely
necessary when $m=n=1$. On the other hand, it is known that
monotonicity is not necessary when $n \geq 3$ (Schmidt) or when
$n=1 $ and $m\geq 2$ (Gallagher). Surprisingly, when $n=2$ the
situation  is unresolved.  We deal with this remaining case and
thereby remove all unnecessary conditions from the classical
Khintchine-Groshev  theorem. This settles a multi-dimensional  analogue of Catlin's
Conjecture.
\end{abstract}

%VB's draft

\section{Introduction}

Throughout, $n \geq 1 $ and $m  \geq 1$ are integers and $\I^{nm}$
is the unit cube $[0,1]^{nm}$ in $ \R^{nm}$. Given a function
$\psi:\N\to\Rp$, let
$ \cA_{n,m}(\psi)$ denote the set of $ \vv X %= (x_{ij})
\in\I^{nm}$  such that
$$|\vv q\vv X+ \vv p |<\psi(|\vv
q|) $$ holds for infinitely many $(\vv p,\vv q) \in \Z^m \times
\Z^n\nz$. Here  $|\cdot|$ denotes the supremum norm, $\vv X =
(x_{ij})$ is regarded as an $n\times m$ matrix and $\vv q$ is
regarded as a row. Thus, $\vv q\vv X $ represents a point in $ \R^m$
given by the  system
$$
q_1x_{1j} + \dots + q_n x_{nj}  \qquad (1 \leq j \leq m) \,
$$
of $m$ real linear forms in $n$ variables. For obvious reasons the
function  $\psi$ is referred to as an \emph{approximating function}
and points in $\cA_{n,m}(\psi)$ are said to be
$\psi$-{\it{approximable}}.

In the case that the approximating function is monotonic, the
classical Khintchine-Groshev theorem provides a beautiful and
strikingly simple criterion for the `size' of $\cA_{n,m}(\psi)$
expressed in terms of $nm$-dimensional Lebesgue measure. The
following is an improved modern version of this fundamental result
-- see \cite{Beresnevich-Dickinson-Velani-06:MR2184760} and
references within. Given a set $X\subset\I^{nm}$, let $|X|$ denote
the $nm$-dimensional Lebesgue measure of $X$.

\begin{kg}
Let $\psi:\N\to\R^+$. Then
  \begin{equation}\label{e:001}
    |\cA_{n,m}(\psi)| = \begin{cases} 0 \
      &\text{if } \quad \sum_{q=1}^\infty q^{n-1}\psi(q)^m<\infty, \\[2ex]
      1 \
      &\text{if } \quad  \sum_{q=1}^\infty q^{n-1}\psi(q)^m=\infty \text{ and $\psi$ is monotonic}.
                  \end{cases}
  \end{equation}
\end{kg}

\noindent The convergence part is reasonably straightforward to
establish and is free from any assumption on $\psi$.
%On exploiting the $\limsup$ nature of the set $\cA_{n,m}(\psi)$, a straightforward
%application of the Borel-Cantelli lemma from probability theory yields the convergence part.
%Note that it is free from any assumption on the approximating
%function $\psi$.
The divergence part constitutes the main substance of the
Khintchine-Groshev theorem and involves the monotonicity assumption
on the approximating function. It is worth mentioning that in the
original statement of the theorem
\cite{Groshev-1938,Khintchine-1924, Khintchine-1926} the stronger
hypothesis that $q^n\psi^m(q)$ is monotonic was assumed.
%%under the stronger condition that
%$q^n\psi^m(q)$ is monotonic, the above statement was first established by Khintchine
%\cite{Khintchine-1924, Khintchine-1926}  in the case $n=1$ and later extended by Groshev \cite{Groshev-1938}
%to arbitrary $n$ and $m$.
 The goal of this article is to
investigate the role of the monotonicity hypothesis in the
Khintchine-Groshev theorem.

In the one-dimensional case ($m=n=1$), it is well know that the
monotonicity hypothesis in the Khintchine-Groshev theorem is
absolutely crucial.  Indeed, Duffin $\&$  Schaeffer
\cite{Duffin-Schaeffer-41:MR0004859} constructed a non-monotonic
function $\psi$ for which $\sum_{q} \psi(q) $ diverges but
$\cA_{1,1}(\psi)$ is of measure zero. In other words the
Khintchine-Groshev theorem is false  without the monotonicity
hypothesis  and the conjectures of Catlin
\cite{Catlin-76:MR0417098} and Duffin $\&$ Schaeffer
\cite{Duffin-Schaeffer-41:MR0004859}  provide appropriate
alternative statements -- see below.  The
 Catlin  and Duffin-Schaeffer conjectures represent two key unsolved  problems in metric number theory.

Beyond the one-dimensional case the situation is very different and
surprisingly incomplete. If $n=1$ and $m \geq 2$, a theorem of
Gallagher \cite{Gallagher-65:MR0188154} implies that the
monotonicity assumption in the Khintchine-Groshev theorem can be
safely removed. Furthermore, the monotonicity assumption can also be
removed if $n \geq 3$,  this time as a consequence of a result of
Schmidt \cite[Theorem 2]{Schmidt-1960} or alternatively a more
general result of Sprind\v{z}uk \cite[\S\,I.5, Theorem
15]{Sprindzuk-1979-Metrical-theory} -- also see
\cite[\S5]{Beresnevich-Bernik-Dodson-Velani-Roth}. It is worth
mentioning that the results of Schmidt and Sprind\v{z}uk are
quantitative and we shall discuss this `stronger' aspect of the theory
at the end of the paper in \S\ref{asypsec}.    Despite the
generality, the theorems of Schmidt and Sprind\v{z}uk leave the case
$n=2$ unresolved and to the best of our knowledge the case is not
covered by any other known result. In this paper we show that the
monotonicity assumption is unnecessary when $n=2$ and thereby
establish the following clear-cut statement that is best possible.

\begin{theorem}\label{t}
Let $\psi:\N\to\R^+$ and $nm>1$. Then
  \begin{equation*}
    |\cA_{n,m}(\psi)| =
      1 \quad  \text{if }  \quad \sum_{q=1}^\infty q^{n-1}\psi(q)^m=\infty .
  \end{equation*}
\end{theorem}

\noindent As already mentioned, Theorem \ref{t} is false when $m
n=1$ and the Catlin conjecture provides the appropriate statement:
\begin{equation*}
    |\cA_{1,1}(\psi)| =
      1
      \quad \text{if } \quad  \sum_{q=1}^\infty \varphi(q) \ \max_{t \geq 1}  \frac{\psi(qt)}{qt} =\infty
      .
 \end{equation*}

%\bigskip

\noindent Here, and throughout, $\varphi$ is  the Euler function.
For further details concerning the above mentioned classical
results and the generalisations of the Catlin and Duffin-Schaeffer
conjectures to linear forms see
\cite{Beresnevich-Bernik-Dodson-Velani-Roth}. Indeed, Theorem
\ref{t} is formally stated as Conjecture A in
\cite{Beresnevich-Bernik-Dodson-Velani-Roth} and is shown to be
equivalent to the linear forms Catlin conjecture.

We shall prove  Theorem \ref{t} by establishing the analogous
statement for an important subset of $\cA_{n,m}(\psi)$. Given two
integer points $\vv p=(p_1,\dots,p_m)\in\Z^m$ and $\vv
q=(q_1,\dots,q_n)\in\Z^n$, let $\gcd(\vv p,\vv q)$ denote the
greatest common divisor of $p_1,\dots,p_m,q_1,\dots,q_n$. We say
that $\vv p$ and $\vv q$ are coprime if $\gcd(\vv p,\vv q)=1$.
Consider the set
\begin{equation*}
\cA'_{n,m}(\psi) := \{\vv X\in \I^{nm}:|\q \vv X +\vv p| < \psi(|\q|)
\begin{array}[t]{l}
  {\text { for infinitely many }} (\vv p,\vv q)\in\Z^m \times \Z^n\nz \\[1ex]
  \text{ with }\gcd(\vv p,\vv q)=1 \}.
\end{array}
\end{equation*}
In view of the coprimeness condition,
 we clearly have that  $
\cA'_{n,m}(\psi)\subset\cA_{n,m}(\psi) $  and  so  Theorem \ref{t}
is a consequence of the following theorem.
\begin{theorem}\label{t1}
Let $\psi:\N\to\R^+$ and $nm>1$. Then
  \begin{equation*}
    |\cA'_{n,m}(\psi)| =
      1
      \quad \text{if } \quad  \sum_{q=1}^\infty q^{n-1}\psi(q)^m=\infty .
 \end{equation*}
\end{theorem}

\noindent As with Theorem \ref{t}, for $ n=1$  the statement of
Theorem \ref{t1}  is due to Gallagher. For $n \geq 3 $ it can be
derived from Schmidt's \cite[Theorem~2]{Schmidt-1960} or
Sprind\v{z}uk's \cite[\S\,I.5, Theorem
15]{Sprindzuk-1979-Metrical-theory}. Furthermore, when $mn=1$ the
Duffin-Schaeffer conjecture provides the appropriate statement:
\begin{equation*}
    |\cA'_{1,1}(\psi)| =
      1
      \quad \text{if } \quad  \sum_{q=1}^\infty \varphi(q) \ \frac{\psi(q)}{q} =\infty .
 \end{equation*}

\vspace*{2ex}

\noindent The proof of Theorem \ref{t1} presented in this paper is
self-contained. In other words, there is little advantage in
restricting the proof to the `unknown' $n=2$ case.
The key to establishing  the theorem is showing that the sets associated with
the natural $\limsup$ decomposition of $\cA'_{n,m}(\psi)$ are
quasi-independent on average --- see Theorem \ref{tqia} below. To the best of our knowledge,
such an independence result is unavoidable when proving positive measure results for $\limsup$ sets. More to the point, the analogue of Theorem \ref{tqia} associated with  the set $\cA_{n,m}(\psi)$ is probably not in general true and it is absolutely necessary to work with the `thinner' set $\cA'_{n,m}(\psi)$.  In particular,  this would explain why Theorem \ref{t} is not in general covered by the result of Schmidt.  Given the nature of his goal, Schmidt was essentially forced to work directly with $\cA_{n,m}(\psi)$.

\vspace*{2ex}

Beyond the above statements, in  \S\ref{multisec}  we discuss the generalizations of Theorems \ref{t} and \ref{t1} within the framework of multivariable approximating functions $\Psi:\Z^n\to\Rp$.  In the final section \S\ref{asypsec}, we discuss the quantitative theory and show that  Theorem \ref{t} can not be deduced from Schmidt's  quantitative theorem.

\section{Preliminaries}

In this section we reduce the proof of  Theorem \ref{t1} to
establishing a quasi-independence on average statement --
Theorem~\ref{tqia} below. We also state various known results that
we appeal to during the course of proving Theorem \ref{tqia}.

We start with an almost trivial  but nevertheless useful
observation.  In Theorem \ref{t1}, there is no loss of generality
in assuming that
\begin{equation}\label{triv}
\psi(h) <   c  \qquad {\rm for \ all \ } h \in \N {\rm  \ and \ }
c
> 0 \ .
\end{equation}
Suppose for the moment that this was not the case and define
$$\Psi : h \to \Psi(h) :=  \min \left\{  c, \psi(h)  \right\}  \
.$$ It is easily verified that if  $\sum h^{n-1}\psi(h)^m $
diverges then $\sum h^{n-1}\Psi(h)^m$ diverges.  Furthermore,
$\cA'_{n,m}(\Psi)  \subset \cA'_{n,m}(\psi) $  and so it suffices
to establish Theorem \ref{t1} for $\Psi$.
\

The next statement is far from being trivial.   It is a
consequence of the main result in
\cite{Beresnevich-Velani-08-Acta-Arith} and reduces the proof of
Theorem \ref{t1} to showing that $ \cA'_{n,m}(\psi)$ is of
positive measure.

\begin{lemma}\label{l5}
For any $n$, $m \geq 1$ and $\psi:\N\to\R^+$,
$$
|\cA'_{n,m}(\psi)|>0\qquad\Longrightarrow\qquad
|\cA'_{n,m}(\psi)|=1.
$$
\end{lemma}

\noindent In order to prove positive measure, we make use of the
following natural representation of  $ \cA'_{n,m}(\psi) $  as  a $\limsup$
set. Given $\delta>0$ and $\vv q\in\Z^n\nz$, let
$$
B(\vv q,\delta):=\{\vv X\in\I^{nm}:|\vv q\vv X+\vv p|<\delta\text{
for some }\vv p\in\Z^m\}
$$
and
$$
B'(\vv q,\delta):=\{\vv X\in\I^{nm}:|\vv q\vv X+\vv p|<\delta
\text{ for some }\vv p\in\Z^m\text{ with }\gcd(\vv p,\vv q)=1\}\,.
$$
Then, it is easily seen that
$$
\cA'_{n,m}(\psi)\, = \, \limsup_{|\vv q|\to\infty}B'(\vv
q,\psi(|\vv q|)).
$$

\noindent The following lemma provides a mechanism for
establishing lower bounds for the measure of $\limsup$ sets. The
statement is a generalisation of the divergent part of the
standard Borel-Cantelli lemma in probability theory, see
\cite[Lemma 5]{Sprindzuk-1979-Metrical-theory}. It is conveniently
adapted for the setup above.

\begin{lemma}\label{l4}
Let $E_k \subset \I^{nm}$ be a sequence of measurable sets such that
$\sum_{k=1}^\infty  |E_k|=\infty $. Then
\begin{equation}\label{e:005}
| \limsup_{k \to \infty} E_k | \; \geq \;  \limsup_{N \to \infty}
\frac{ \left( \sum_{s=1}^{N} |E_s| \right)^2 }{ \sum_{s, t =
1}^{N} |E_s \cap E_t | }  \  \  \ .
\end{equation}

\end{lemma}

\noindent In view of Lemma~\ref{l4}, the  desired statement $
|\cA'_{n,m}(\psi)|>0 $ will follow on showing that the sets
$B'_{\vv q}(\psi):=B'(\vv q,\psi(|\vv q|))$   are {\rm
quasi-independent on average} and that the sum of their measures diverges. Formally, we  shall prove the
following statement.

\begin{theorem}[quasi-independence on average]\label{tqia}
Let  $nm>1$ and $\psi:\N\to\R^+$ satisfy  $\psi(h) < 1/2 $ for all $h \in \N$ and  $\sum_{h=1}^{\infty} h^{n-1}\psi(h)^m = \infty$. Then
\begin{equation}\label{divsum}
 \ds\sum_{\vv q  \in \Z^n\nz}  |B'_{\vv q}(\psi) | \ = \ \infty  \ ,
\end{equation}
and there exists
a constant $C
> 1$ such that for  $N$ sufficiently large,
\begin{equation}\label{qia}
 \ds\sum_{|\vv q_1|\le N,\ |\vv q_2|\le N}|B'_{\vv q_1}(\psi)
 \cap B'_{\vv q_2}(\psi)| \ \leq \ C \,
 \displaystyle \Big( \sum_{ |\vv q_1|\le N}|B'_{\vv q_1}(\psi)| \Big)^{2}   \ .
\end{equation}
\end{theorem}

\noindent  The upshot of the above discussion is that
$$
{\rm Theorem \ \ref{tqia}    \qquad  \Longrightarrow  \qquad
Theorem  \ \ref{t1}  }  \ .
$$

\noindent In order to establish the quasi-independence on average
statement, we will make use of  the following results concerning
the sets $B(\vv q,\delta)$.

\begin{lemma}\label{l2}
Let $n$, $m \geq 1$ and  let $\vv q_1,\vv q_2\in \Z^n\nz$ and
$\delta_1,\delta_2\in(0,1/2)$. Then
\begin{equation}\label{add3}
|B(\vv q_1,\delta_1)|=(2\delta_1)^m\,
\end{equation}
 and
\begin{equation}\label{add2}
|B(\vv q_1,\delta_1)\cap B(\vv q_2,\delta_2)|=|B(\vv
q_1,\delta_1)|\cdot|B(\vv q_2,\delta_2)|\,  \ \quad\text{ if\/    \ \
$\vv q_1\nparallel\vv q_2$}.
\end{equation}
\end{lemma}

\noindent The notation $\vv q_1\|\vv q_2$ means that $\vv q_1$ is
parallel to $\vv q_2$. The lemma is a consequence of Lemmas~8 and 9
in \cite{Sprindzuk-1979-Metrical-theory} and  implies that the
sets $B(\vv q_i,\delta_i)$, $i=1,2$ are pairwise independent for non-parallel
vectors. The following statement is an analogue of Lemma~\ref{l2}
for the sets $B'(\vv q,\delta)$ with $n=1$.

\begin{lemma}\label{lem9-}
Let $n=1$ and $ m \geq 1$. There is a constant $C>0$ such that for
$\delta_1,\delta_2\in(0,1/2)$ and any distinct  $q_1,q_2\in\N$
\begin{equation}\label{a2++}
  |B'(q_1,\delta_1)| = (2\delta_1)^m\prod_{p|q_1}(1- p^{-m})
  \vspace*{-2ex}
\end{equation}
 and
\begin{equation}\label{a2+}
  |B'(q_1,\delta_1)\cap B'(q_2,\delta_2)| \; \le  \; C \, (\delta_1 \, \delta_2)^m\,.
\end{equation}

\noindent The product in $(\ref{a2++})$ is  over prime divisors $p$ of
$q_1$ and is defined to be one if $q_1=1$.
\end{lemma}

\noindent In the case $m=1$, the inequality given by  (\ref{a2+})
follows from equation~(36) in
\cite{Sprindzuk-1979-Metrical-theory}. In the case $m\ge 2$,  the
inequality   follows from equation~(10) in
\cite{Gallagher-65:MR0188154}.   Finally, the equality given by
(\ref{a2++}) is established within the proof of Lemma~1 in
\cite{Gallagher-65:MR0188154}. Note that when  $m \geq 2 $,  the
product term in (\ref{a2++}) is comparable to a constant and the
lemma implies that the sets $B'(q,\delta)$  are pairwise
quasi-independent.

\vspace{1ex}

We bring this section of preliminaries  to an end by stating a
counting result that can be found in
\cite[p.\,39]{Sprindzuk-1979-Metrical-theory}.  Throughout, the
symbols $\ll$ and $\gg$ will be used to indicate an inequality
with an unspecified positive multiplicative constant.   If $a \ll
b$ and $ a \gg b $ we write $a\asymp b$, and say that the
quantities $a$ and $b$ are comparable.

\begin{lemma}\label{l3}
Let $h$ be a positive integer. Then
\begin{equation}\label{e:004}
    \hspace{-10ex}\sum_{\hspace{0ex}\vv q \in \Z^n  \; : \   |\vv q|=h, \ \gcd(\vv
q)=1} \!\! \hspace{-6ex}  1 \ \hspace{5ex} \asymp   \ \
\left\{\begin{array}{cl}
                       \varphi(h)  & \text{\rm if }n=2  \\[1ex]
                       h^{n-1}  & \text{\rm if } n\ge3  \ ,
                     \end{array}
 \right.
\end{equation}
where the implied constants are independent of  $h$.
\end{lemma}

\section{Quasi-independence on average}

We have seen that establishing quasi-independence on average as
stated in Theorem \ref{tqia}
 lies at the heart of Theorem
\ref{t1}.  The proof of Theorem \ref{tqia} splits naturally into
establishing various key measure estimates.

\subsection{Measure of $B'(\vv q,\delta)$ and $B'(\vv
q_1,\delta_1)\cap B'(\vv q_2,\delta_2)$}

The goal of this section is to extend the measure estimates of Lemma
\ref{lem9-} beyond the $n=1$ case.   Given $\delta>0$, $\vv q\in\Z^n\nz$ and $\vv
p\in\Z^m$,  let
$$
B(\vv q,\vv p,\delta):=\{\vv X\in\I^{nm}:|\vv q\vv X+\vv p|<\delta\}.
$$
The following observation  will enable us to determine the precise measure of $B'(\vv q,\delta)$.

\begin{lemma}\label{lem5}
Let $n, m \geq 1 $ and  let $\vv q  \in \Z^n\nz$ and  $\delta
\in(0,1/2)$. Then, for any $l|\gcd(\vv q)$
\begin{equation}\label{e:007}
\sum_{\vv p\in\Z^m}|B(\vv q,l\vv p,\delta)|=\left(\frac{2\delta}{l}\right)^m.
\end{equation}
\end{lemma}

\begin{proof}
 Given that $\delta<1/2$, it is easily seen that the sets $B(\vv q,l\vv p,\delta) $ as  $\vv p$ varies are pairwise disjoint. Therefore, in view of (\ref{add3}) the lemma  follows on showing that
\begin{equation}\label{e:010}
\bigcup_{\vv p\in\Z^m}B(\vv q,l\vv p,\delta)=B(\vv q/l,\delta/l)  \, .
\end{equation}
By definition, $\vv X$ belongs to the left hand side of (\ref{e:010}) if and only if there is a $\vv
p\in\Z^m$ such that $|\vv q\vv X+l\vv p|<\delta$. Since $l$ divides $\vv q$ this is equivalent to
$|(\vv q/l)\vv X+\vv p|<\delta/l$ which, by definition  is equivalent to the statement that  $\vv X$ belongs to the right hand side of (\ref{e:010}).
\end{proof}

In the case $n=1$, the following statement
reduces to (\ref{a2++}) of Lemma \ref{lem9-}.

\begin{lemma}\label{nbig1}
Let $n, m \geq 1 $ and  let $\vv q  \in \Z^n\nz$ and  $\delta
\in(0,1/2)$. Then,
\begin{equation}\label{e:015}
|B'(\vv q,\delta)|=(2\delta)^m\prod_{p|d}(1-p^{-m})\, .
\end{equation}
The product is over prime divisors $p$ of $d:= \gcd(\vv q)$ and is
defined to be one if $d=1$.
\end{lemma}

\begin{proof}
Recall the following well known identities for the M\"obius function $\mu$.
$$
\sum_{l|d}\mu(l)=\left\{\begin{array}{cc}
                          0 & \text{if } \ d>1  \\
                          1 & \text{if } \ d=1
                        \end{array}
\right.\qquad\text{and}\qquad \sum_{l|d}\frac{\mu(l)}{l^m}=\prod_{p|d}\left(1- p^{-m} \right)   \ .
$$
Also, for any set $A$ we use the convention that  $A\times 1:=A$ and $A\times 0:=\emptyset$. Then,
\begin{equation}\label{add1}
B'(\vv q,\delta)=\bigcup_{\vv p\in\Z^m}B(\vv q,\vv p,\delta)  \  \sum_{l|d}\mu(l)\,.
\end{equation}
Since $\delta<1/2$, the sets in the above union  do not overlap and so
$$
\begin{array}{rcl}
 \displaystyle |B'(\vv q,\delta)| & \stackrel{\eqref{add1}}{=} & \displaystyle\sum_{\vv p\in\Z^m}|B(\vv q,\vv p,\delta)|   \ \sum_{l|d}\mu(l)\\[4ex]
 & = & \displaystyle \sum_{l|d=\gcd(\vv q)}\mu(l)  \
 \sum_{\vv p\in\Z^m}|B(\vv q,l\vv p,\delta)|\\[4ex]
 & \stackrel{\eqref{e:007}}{=} & \displaystyle  \sum_{l|d}\mu(l)\left(\frac{2\delta}{l}\right)^m
  =  \displaystyle (2\delta)^m\prod_{p|d}\left(1-\frac{1}{p^m}\right).
\end{array}
$$
\vspace*{-3ex}
\end{proof}

\

The following is a consequence of examining the product term in
 Lemma \ref{nbig1}.

\begin{lemma}\label{lem8}
Let $n \geq 1 $ and  let $\vv q  \in \Z^n\nz$,  $d:= \gcd(\vv q)$
and $\delta \in(0,1/2)$.  If $m=1$, then
\begin{equation}\label{e:018}
|B'(\vv q,\delta)|= 2\delta\ \frac{\varphi(d)}{d}\,.
\end{equation}
If $m>1$, then
\begin{equation}\label{e:017}
\frac6{\pi^2}\,(2\delta)^m \le |B'(\vv q,\delta)|\le (2\delta)^m\,.
\end{equation}

\end{lemma}

\begin{proof}
In the case $m>1$, we trivially have that
\begin{equation}
\label{priv}
1\ge
\prod_{p|d}\big(1-p^{-m}\big)>\prod_{p}\big(1-p^{-2}\big)=\frac{1}{\zeta(2)}=\frac{6}{\pi^2}\,.
\end{equation}
Therefore (\ref{e:015}) implies (\ref{e:017}). In the case $m=1$, we
have that
$
\prod_{p|d}\big(1-p^{-m}\big)=\varphi(d)/d  .
$
Therefore (\ref{e:015}) implies (\ref{e:018}).
\end{proof}

We now turn our attention to estimating the measure of the pairwise
intersection between the sets $B'(\vv q,\delta)$. In the case $n=1$,
the following statement coincides with (\ref{a2+}) of Lemma
\ref{lem9-}.

\begin{lemma}\label{lem9}
Let $n, m \geq 1$. There is a constant $C>0$ such that for
$\delta_1,\delta_2\in(0,1/2)$ and   $\vv q_1,\vv q_2\in\Z^n\nz$
satisfying $\vv q_1\not=\pm\vv q_2$
\begin{equation}\label{a2}
  |B'(\vv q_1,\delta_1)\cap B'(\vv q_2,\delta_2)| \le C(\delta_1\delta_2)^m\,.
\end{equation}
\end{lemma}

\begin{proof}
In view of the fact that $B'(\vv q,\delta)\subset B(\vv q,\delta)$ and Lemma~\ref{l2},  we only need to deal with
the situation when  $\vv q_1$ and $\vv
q_2$ are parallel. Then, it follows that there exists $\vv q\in\Z^n$
with $\gcd(\vv q)=1$ and two different positive integers $k_1,k_2$
such that $\vv q_1=k_1\vv q$ and $\vv q_2=\pm k_2\vv q$. Without
loss of generality,  assume that $\vv q_2=k_2\vv q$.

%
%The case when $n=1$ is covered by Lemma~\ref{lem9-}. Furthermore, in
%view of (\ref{e:006}) and Lemma~\ref{l2} the case when $\vv
%q_1\nparallel\vv q_2$ is also covered.  Thus, we need to deal with
%the situation when $n \geq 2 $ and the vectors   $\vv q_1$ and $\vv
%q_2$ are parallel. Then, it follows that there exists $\vv q\in\Z^n$
%with $\gcd(\vv q)=1$ and two different positive integers $k_1,k_2$
%such that $\vv q_1=k_1\vv q$ and $\vv q_2=\pm k_2\vv q$. Without
%loss of generality,  assume that $\vv q_2=k_2\vv q$.

Let $\vv X\in B'(\vv q_1,\delta_1)\cap B'(\vv q_2,\delta_2)$. By
definition, there are integer points $\vv p_1,\vv p_2\in\Z^m$ such
that $|\vv q_i\vv X+\vv p_i|<\delta_i$ and $\gcd(\vv p_i,\vv q_i)=1$
for $i=1,2$. Equivalently we have that
\begin{equation}\label{e:023}
\left\{\begin{array}{rl}
    |k_1\vv q\vv X+\vv p_1|<\delta_1,&\quad \gcd(k_1,\vv
    p_1)=1\,,\\[1ex]
    |k_2\vv q\vv X+\vv p_2|<\delta_2,&\quad \gcd(k_2,\vv p_2)=1\,.
    \end{array}
    \right.
\end{equation}
Consider the transformation
\begin{equation}\label{add4}
    T_{\vv q}:\I^{mn}\to\I^m\ :\ \vv X\mapsto \vv q\vv X\bmod 1\,.
\end{equation}
It is  readily verified that
\begin{equation}\label{add5}
B'(\vv q_1,\delta_1)\cap B'(\vv q_2,\delta_2)\subseteq T_{\vv
q}^{-1}\big(B'(k_1,\delta_1)\cap B'(k_2,\delta_2)\big).
\end{equation}
The transformation $T_{\vv q}$ is measure preserving; i.e. for
any measurable set $A\subset\I^m$ we have that $|T^{-1}_{\vv
q}(A)|=|A|$ -- see equation (48) in
\cite{Sprindzuk-1979-Metrical-theory}. Therefore, by (\ref{add5}) we
have that
\begin{equation}\label{add6}
    |B'(\vv q_1,\delta_1)\cap B'(\vv q_2,\delta_2)|\le
    |B'(k_1,\delta_1)\cap B'(k_2,\delta_2)|\,.
\end{equation}
Applying Lemma~\ref{lem9-} to (\ref{add6}) completes the proof of
the lemma.
\end{proof}

\subsection{Measure of $B'_{\vv q}(\psi)$ on average}

\begin{lemma}\label{lem10}
Let $nm>1$ and $\psi(h)<1/2$ for all $h\in\N$. Then  with  $\vv q\in\Z^n\nz$ and $ N \in \N$,
\begin{equation}\label{e:019}
    \sum_{ |\vv q|\le N}|B'_{\vv q}(\psi)| \ \asymp \
    \sum_{h=1}^N h^{n-1}\psi(h)^m   \ .
\end{equation}
\end{lemma}

\begin{proof}
Naturally, the proof makes use of Lemma~\ref{lem8} and therefore splits
into two cases: $m>1$ and $m=1$.
We begin by considering the easy case
$m>1$. By (\ref{e:017}) and the fact that the number of
integer points $\vv q\in\Z^n$ with $|\vv q|=h$ is comparable to $  h^{n-1}$, we have that
$$
\begin{array}{rcl}
    \displaystyle \sum_{\vv q\in\Z^n\nz,\ |\vv q|\le N}|B'_{\vv q}(\psi)| & \asymp
    & \displaystyle \sum_{\vv q\in\Z^n\nz,\ |\vv q|\le N}\psi(|\vv q|)^m\\[4ex]
    & \asymp & \displaystyle \sum_{h=1}^N  \  \sum_{|\vv q|=h}\psi(|\vv q|)^m \\[2ex]
    & \asymp &
    \displaystyle \sum_{h=1}^N h^{n-1}\psi(h)^m\,.
\end{array}
$$
This establishes (\ref{e:019}) in the case $m>1$.

We proceed with the case $m=1$. By
(\ref{e:018}), it follows that
\begin{equation}\label{e:020}
    \begin{array}[b]{rcl}
  \displaystyle\sum_{\vv q\in\Z^n\nz,\ |\vv q|\le N}|B'_{\vv q}(\psi)|
  & = &\displaystyle \sum_{h=1}^N \ \sum_{\vv q\in\Z^n,\ |\vv q|=h}|B'_{\vv q}(\psi)|
  \\[2ex]
   & \asymp& \displaystyle\sum_{h=1}^N  \  \sum_{\vv q\in\Z^n,\ |\vv
   q|=h}\frac{\varphi(d)}{d}\,\psi(h) \hspace{10ex}  d:= \gcd(\vv q)  \\[4ex]
 & = &\displaystyle \sum_{h=1}^N\psi(h)\ \ \sum_{d|h}\frac{\varphi(d)}{d}  \hspace{2ex}
\sum_{|\vv q'|=h/d,\ \gcd(\vv q')=1}\hspace{-4ex}  1   \hspace{2ex} .

\end{array}
\end{equation}

\noindent To analyze (\ref{e:020}) we consider  $n > 2 $ and $n=2$ separately.  Recall, that  $nm>1$ is a hypothesis within the statement of the lemma and so  $n=1$ is barred.

\noindent{\em Subcase $n>2$\,}: By Lemma~\ref{l3}, it follows from
(\ref{e:020}) that
$$
\begin{array}{rcl} \label{clear}
   \displaystyle \sum_{\vv q\in\Z^n\nz,\ |\vv q|\le N}|B'_{\vv q}(\psi)| &\asymp
   & \ds \sum_{h=1}^N\psi(h)\ \sum_{d|h}\frac{\varphi(d)(h/d)^{n-1}}{d}
   \\[2ex]
   &\asymp  & \ds \sum_{h=1}^N h^{n-1}\psi(h) \
   \sum_{d|h}\frac{\varphi(d)}{d^n}\,.
\end{array}
$$
This together with the fact that
$$
1\  \le \ \sum_{d|h}\frac{\varphi(d)}{d^n} \ \le \
\sum_{d=1}^\infty\frac{1}{d^2} \ = \ \frac{\pi^2}{6}\ ,
$$
 yields (\ref{e:019}).

\medskip

\noindent\noindent{\em Subcase $n=2$\,}: By Lemma~\ref{l3}, it follows from
(\ref{e:020}) that
\begin{equation}\label{e:021}
    \displaystyle \sum_{\vv q\in\Z^n\nz,\ |\vv q|\le N}|B'_{\vv q}(\psi)| \  \asymp\ \sum_{h=1}^N\psi(h)\
     \sum_{d|h}\frac{\varphi(d)\varphi(h/d)}{d}\ = \
     \sum_{h=1}^N\psi(h)f(h)\ ,
\end{equation}
where
$$
\begin{array}{rcl}
 f(h) &:= & \ds \sum_{d|h}\frac{\varphi(d)\varphi(h/d)}{d}\\[4ex]
      & = & \ds \sum_{d|h}\varphi(h/d)\sum_{l|d}\frac{\mu(l)}{l}\\[4ex]
      & = & \ds \sum_{l|h}\frac{\mu(l)}{l}\sum_{k|h/l}\varphi(k)\\[4ex]
      & = & \ds h\sum_{l|h}\frac{\mu(l)}{l^2}\qquad\text{since \ }   \ \sum_{k|d}\varphi(k)=d\\[4ex]
 & = &  \ds h\prod_{p|h}\left(1- p^{-2} \; \right)  \ .
\end{array}
$$
Therefore,  by (\ref{priv})  we have that
$$ \frac{6}{\pi^2} \, h \ \le \
f(h) \ \le \  h   \qquad {\rm for \ all \ }  h\in\N   \, .
$$
\noindent This combined with  (\ref{e:021}) yields
(\ref{e:019}) with $m=1$ and $n=2$.
\end{proof}

%Recall,
%\sum_{ |\vv q|\le N}|B'_{\vv q}(\psi)| \ \asymp \
%    \sum_{h=1}^N h^{n-1}\psi(h)^m

\subsection{Measure of $B'_{\vv q_1}(\psi)\cap B'_{\vv q_2}(\psi)$ on average}

\begin{lemma}\label{lem11}
Let  $nm>1$, $\psi(h) < 1/2 $ for all $h \in \N$ and  $\sum h^{n-1}\psi(h)^m = \infty$. Then with  $\vv q_1,\vv q_2\in\Z^n\nz$ and $N $ sufficiently large,
\begin{equation}\label{e:019+}
    \ds\sum_{|\vv q_1|\le N,\ |\vv q_2|\le N}|B'_{\vv q_1}(\psi)\cap B'_{\vv q_2}(\psi)|
   \ll \left(\sum_{h=1}^N h^{n-1}\psi(h)^m\right)^2   \ .
\end{equation}
\end{lemma}

\begin{proof}
To begin with we separate out the diagonal term from the double sum
in (\ref{e:019+}) and treat it separately as follows. Since the sum
$\sum h^{n-1}\psi(h)^m$ diverges, there exists  a positive integer
$N_0 $ such that $\sum_{h=1}^N h^{n-1}\psi(h)^m>1$ for all $N>N_0$.
Then, by Lemma~\ref{lem10} it follows that for  $N>N_0$
\begin{equation*}%\label{diag}
\begin{array}[b]{rcl}
\ds\sum_{\stackrel{\scriptstyle |\vv q_1|\le N,\ |\vv q_2|\le
N}{\rule{0pt}{2ex}\vv q_2=\pm\vv q_1}}
%\ds\sum_{|\vv q_1|\le N,\ \vv q_2 =\pm\vv q_1}
|B'_{\vv q_1}(\psi)\cap B'_{\vv q_2}(\psi)| & = &
     \displaystyle 2\sum_{|\vv q_1|\le N}|B'_{\vv q_1}(\psi)|\\[3ex]
       & \ll &\displaystyle \sum_{h=1}^N h^{n-1}\psi(h)^m \\[4ex]
      & < & \displaystyle\left(\sum_{h=1}^N h^{n-1}\psi(h)^m\right)^2   \ .
\end{array}
\end{equation*}
To complete the proof of the lemma,  we  obtain a similar estimate for the remaining part of the double sum. In view of  Lemma~\ref{lem9}, it follows that
$$
\begin{array}{rcl}
  \ds\sum_{\stackrel{\scriptstyle |\vv q_1|\le N,\ |\vv q_2|\le N}{\rule{0pt}{2ex}\vv
q_2\not=\pm\vv q_1}}|B'_{\vv q_1}(\psi)\cap B'_{\vv q_2}(\psi)|\ \ &  =
 &\ds \sum_{h=1}^N\ \sum_{l=1}^N\ \ \sum_{\stackrel{\scriptstyle |\vv q_1|=h,\ |\vv q_2|=l}{\vv
q_2\not=\pm\vv q_1}}|B'_{\vv q_1}(\psi)\cap B'_{\vv q_2}(\psi)| \\[3ex]
    &\ll &\ds \sum_{h=1}^N\ \sum_{l=1}^N\ \ \sum_{|\vv q_1|=h,\ |\vv q_2|=l} \psi(|\vv q_1|)^m\cdot \psi(|\vv q_2|)^m \\[3ex]
   & = &\ds \sum_{h=1}^N \ \sum_{l=1}^N\psi(h)^m \cdot \psi(l)^m \ \ \sum_{|\vv q_1|=h}1\ \ \sum_{|\vv q_2|=l}1 \\[3ex]
    &\ll &\ds \sum_{h=1}^N\ \sum_{l=1}^N h^{n-1}\psi(h)^m\cdot l^{n-1}\psi(l)^m\\[3ex]
    &\ll &\ds \left(\sum_{h=1}^Nh^{n-1}\psi(h)^m\right)^2\,.
\end{array}
$$
\vspace*{-2ex}
\end{proof}

\subsection{The finale}

Let  $nm>1$, $\psi(h) < 1/2 $ for all $h \in \N$ and  $\sum h^{n-1}\psi(h)^m = \infty$.  On combining
Lemmas~\ref{lem10} and \ref{lem11}, we have that for $\vv
q_1,\vv q_2\in \Z^n\nz$ and $N \in \N $ sufficiently large
$$
 \ds\sum_{|\vv q_1|\le N,\ |\vv q_2|\le N}|B'_{\vv q_1}(\psi)
 \cap B'_{\vv q_2}(\psi)| \ \ll \  \displaystyle \Big( \sum_{ |\vv q_1|\le N}|B'_{\vv q_1}(\psi)| \Big)^{2}   \ .
$$
Furthermore, an obvious implication of Lemma~\ref{lem10} is that
$$
\ds\sum_{\vv q  \in \Z^n\nz}  |B'_{\vv q}(\psi) | \ = \ \infty  \ .
$$
The above  are precisely the expressions given by (\ref{divsum}) and (\ref{qia}) and  thereby
completes the proof Theorem \ref{tqia}.

%\section{Generalisations and further problems \label{general}}

\section{ The multivariable theory \label{multisec} }

Given a vector  $\vv q \in \Z^n$, the  approximating function
$\psi:\N\to\Rp$  assigns a quantity $\psi(|\vv q|)$ that is
dependent on the supremum norm of $\vv q$. Clearly, a natural and
desirable generalisation is to consider multivariable approximating
functions $\Psi:\Z^n\to\Rp$  and their associated sets
$\cA_{n,m}(\Psi)$ and  $\cA'_{n,m}(\Psi)$  of $\Psi$-approximable
points. When the argument of $\Psi $ is restricted to the supremum
norm these sets are precisely  the sets of $\psi$-approximable
points considered above. For the sake of clarity, given
$\Psi:\Z^n\to\Rp$ let
\begin{equation*}
\cA'_{n,m}(\Psi) := \{\vv X\in \I^{nm}:|\q \vv X +\vv p| < \Psi(\q)
\begin{array}[t]{l}
  {\text { for infinitely many }} (\vv p,\vv q)\in\Z^m \times \Z^n\nz \\[1ex]
  \text{ with }\gcd(\vv p,\vv q)=1 \}.
\end{array}
\end{equation*}

\noindent Modifying the proof of Theorem~\ref{t1} in the obvious
manner, leads to the following statement.

\begin{theorem}\label{t3}
Let $\Psi:\Z^n\to\R^+$ and $m>1$. Then
  \begin{equation*}
    |\cA'_{n,m}(\Psi)| = 1 \quad  \text{if }  \quad  \sum_{\vv q\in\Z^n\nz}  \Psi(\vv
    q)^m=\infty  \ .
\end{equation*}
\end{theorem}

%
%As with Theorem~\ref{t1}, the proof of Theorem~\ref{t3} reduces to
%establishing the  analogue of
%Theorem~\ref{tqia} -- in particular, on showing that
%\begin{equation}\label{qia2}
% \ds\sum_{|\vv q_1|\le N,\ |\vv q_2|\le N}|B'(\vv q_1,\Psi(\vv q_1))
% \cap B'(\vv q_2,\Psi(\vv q_2))| \ \ll \
% \displaystyle \Big( \sum_{ |\vv q_1|\le N}|B'(\vv q_1,\Psi(\vv q_1))| \Big)^{2}   \ .
%\end{equation}
%However, since we are assuming that $m >1 $ the proof
%is much simpler. The reason for this is that
%(\ref{e:017}) and (\ref{a2}) actually imply pairwise quasi-independence for
%the off-diagonal terms  ($\vv q_2\not=\pm\vv q_1$) of (\ref{qia2}).
%Thus,  there is no need to `average' as in the proof of Theorem~\ref{t1}.

As with Theorem~\ref{t1}, the proof of Theorem~\ref{t3} reduces to
establishing the  analogue of
Theorem~\ref{tqia} -- in particular, on showing that
\begin{equation}\label{qia2}
 \ds\sum_{|\vv q_1|\le N,\ |\vv q_2|\le N}|B'(\vv q_1,\Psi(\vv q_1))
 \cap B'(\vv q_2,\Psi(\vv q_2))| \ \ll \
 \displaystyle \Big( \sum_{ |\vv q_1|\le N}|B'(\vv q_1,\Psi(\vv q_1))| \Big)^{2}   \ .
\end{equation}
However, since we are assuming that $m >1 $ the proof
is much  simpler. The reason for this is that
(\ref{e:017}) and (\ref{a2}) actually imply pairwise quasi-independence for
the off-diagonal terms; i.e. $
|B'(\vv q_1,\Psi(\vv q_1)) \cap B'(\vv q_2,\Psi(\vv q_2))|  \ \ll  \   |B'(\vv q_1,\Psi(\vv q_1))| \, | B'(\vv q_2,\Psi(\vv q_2))|
$ for $\vv q_2\not=\pm\vv q_1. $

Our final result is a straightforward consequence of
Theorem~\ref{t3}.

\begin{theorem}\label{t4}
Let $\Psi:\Z^n\to\R^+$ and $m>1$. Then
  \begin{equation}\label{more1}
    |\cA_{n,m}(\Psi)| = 1 \quad  \text{if }  \quad  \sum_{\vv q\in\Z^n\nz} \Psi(\vv
    q)^m=\infty  \ .
\end{equation}
\end{theorem}

\noindent The condition $m > 1 $ cannot in general be removed from either Theorem~\ref{t3}
or Theorem~\ref{t4}.  For
a concrete counterexample see
\cite[\S5]{Beresnevich-Bernik-Dodson-Velani-Roth}. Note that the
statement of Theorem~\ref{t4} was previously established by
Sprind\v{z}uk \cite{Sprindzuk-1979-Metrical-theory} for
approximating functions obeying additional constraints. For example,
Theorem~14 in \cite{Sprindzuk-1979-Metrical-theory} is applicable to
$\Psi$ that vanish on non-primitive $\vv q\in\Z^n$. Our
Theorem~\ref{t4} carries no restrictions on $\Psi$ and so
is best possible.

For the sake of completeness, we mention that Theorems~\ref{t3} and
\ref{t4} are formally stated as Conjectures~B and~C in
\cite{Beresnevich-Bernik-Dodson-Velani-Roth}.
Furthermore, the Mass Transference Principle
of \cite{Beresnevich-Velani-06:MR2259250} and the `slicing' technique of \cite{Beresnevich-Velani-06:MR2264714} together with Theorem~\ref{t4} establishes the general Hausdorff measure version
of Catlin's conjecture under the assumption that $m \geq 2$ -- see
Conjecture~G in \cite{Beresnevich-Bernik-Dodson-Velani-Roth}.

\section{The quantitative theory \label{asypsec}}

Let   $\Psi:\Z^n\to\R^+$.  Given  $\vv X\in \I^{nm}$ and $h \in \N $,    let
$$
{\cal N}(\vv X,h) \ := \ \# \left\{ (\vv p,\vv q)\in\Z^m \times \Z^n\nz \, : \, |\q \vv X  + \vv p |< \Psi(\q) \text{ with }  |\vv q| \leq h  \right\} \, .
$$

\noindent In view of  Theorem \ref{t4}, if $m > 1$ and  $\sum  \Psi(\q)^m $ diverges then
for almost all $\vv X$ we have that ${\cal N}(\vv X,h) \! \to \! \infty $
as $ h \to \infty$. An obvious question now arises: can we say
anything more precise about the behavior of
${\cal N}(\vv X,h)$?  To some extent,  the following remarkable
statement  provides the answer.  Throughout, $d(h)$ denotes the number of divisors of $h$.

\begin{thbs} Let $\ve > 0 $ be arbitrary. Let $\Psi:\Z^n\to\R^+$ and  write
$$
\Phi(h):=  \!\!\!\! \sum_{\vv q\in\Z^n\nz,\ |\vv q|\le h}  \!\!\!\!   (2\Psi(\vv
q))^m\qquad\text{and}\qquad\chi(h):= \!\!\!\!  \sum_{\vv q\in\Z^n\nz,\ |\vv q|\le h}  \!\!\!\! (2\Psi(\vv
q))^m  \ d(\gcd(\vv q))  \, .
$$
Then, for almost all $\vv X\in \I^{nm}$
\begin{equation}
\label{asympformula}
{\cal N}(\vv X,h) \, = \, \Phi(h) \, +  \, O\Big(\chi^{1/2}(h)\log^{3/2+\ve}\chi(h)\Big)  \, .
\end{equation}
\end{thbs}

\vspace{1ex}

\noindent  The above form of the theorem is in line with the setup considered in this paper.  Schmidt \cite{Schmidt-1960} actually proves a more general statement in which each of the $m$ linear forms associated with the system $\vv q\vv X$  are allowed to be approximated with different approximating functions.

Although not explicitly mentioned in the statement of Schmidt's theorem, we may as well assume that
$\sum  \Psi(\q)^m $ diverges.  Otherwise, a straightforward application of the  Borel-Cantelli
Lemma implies  that $\lim_{h \to \infty}{\cal N}(\vv X,h)  < \infty $
for almost all $\vv X$ and the theorem  is of little interest.  However, it is not the case that if the sum $\sum  \Psi(\q)^m $ diverges then Schmidt's theorem implies that $\lim_{h \to \infty}{\cal N}(\vv X,h)  = \infty $ for almost all $\vv X$; that is to say that  Schmidt's theorem does not in general imply that $|\cA_{n,m}(\Psi)| = 1$. The reason for this is simple. The Duffin-Schaeffer counterexample and the counterexample alluded to in \S\ref{multisec} above   imply that the full measure statement is not in general true when $n=m=1$ or when $m >1 $. Note that these cases are not excluded from Schmidt's theorem and so for the corresponding  counterexamples we must have that the error term in (\ref{asympformula}) outweighs the main term.  We now show that this conclusion is also true when $n=2$ for certain approximating functions with argument restricted to the supremum norm.   Thus, Schmidt's theorem does not imply the theorems established in this paper.

With Theorem \ref{t} in mind,  let $\Psi(\vv q) = \psi(|\vv q|)$ in the above and assume throughout that
\begin{equation}
\label{divslv}
\sum_{q=1}^\infty q^{n-1}\psi(q)^m=\infty   \ .
\end{equation}
%Then, with reference to Schmidt's theorem it is readily verified that when  $n \geq 3$ the  main term  $ \Phi(h) \asymp  \sum_{q=1}^h q^{n-1}\psi(q)^m$ dominates the error term in (\ref{asympformula}).  Thus for $n \geq 3$, Schmidt's theorem not only implies Theorem \ref{t} but provides the stronger quantitative statement. However, when $n=2$ the situation is rather different.

\begin{lemma}  \label{exmp1}
Let $n=2$ and $F:\Rp\to\Rp$ be an increasing function. Then there exists
an approximating function $\psi:\N\to\Rp$ satisfying  $(\ref{divslv})$ such that $\psi$ is monotonic on its support and
\begin{equation}\label{vb2}
    \chi(h)\ge F(\Phi(h))\qquad\text{for all sufficiently large }h.
\end{equation}
\end{lemma}

\noindent\textit{Remark.} Note that for any $\psi$  satisfying the  divergent condition (\ref{divslv}),  we trivially have that the main term $\Phi(h)$ in Schmidt's theorem tends to infinity  as $h \to \infty$. The lemma shows that there exists $\psi$ satisfying  (\ref{divslv}) for which the error term can be made as large as we please compared to the main term.
 For example, with  $F(x) := \exp(2x)$ there exists  $\psi$ for which the
error term  is eventually exponentially
larger than the main term.  Clearly, for such $\psi$ Schmidt's theorem does not enable us to conclude that
 $\lim_{h \to \infty}{\cal N}(\vv X,h)  = \infty $ for almost all $\vv X$ and therefore does not  imply Theorem \ref{t}.

\bigskip

\begin{proof}
Given $l\in\N$, it is easily seen that the number of points $\vv
q\in\Z^2$ such that $|\vv q|=l$ is equal to $8l-4$. With $n=2$,  it follows that
\begin{equation}\label{vb-1}
\Phi(h):=\sum_{|\vv q|\le h}\Psi(\vv q)^m=\sum_{l=1}^h\sum_{|\vv
q|=l}\psi(l)^m\le 8 \sum_{l=1}^h l\psi(l)^m
\end{equation}
 and
\begin{equation}\label{vb-2}
\begin{array}[b]{rcl}
\chi(h) & := & \displaystyle \sum_{|\vv q|\le h}\Psi(\vv
q)^md(\gcd(\vv q)) \ =\ \displaystyle \sum_{l=1}^h\sum_{|\vv
q|=l}\psi(l)^md(\gcd(\vv
 q)) \\[3ex]
 & = & \displaystyle \sum_{l=1}^h\psi(l)^m\sum_{v|l}d(v)\sum_{|\vv q'|=l/v,\ \gcd(\vv q')=1}1 \\[3ex]
 & \geq & \displaystyle
 \sum_{l=1}^h\psi(l)^m\sum_{v|l}d(v)\varphi(l/v)=\sum_{l=1}^h\psi(l)^m
 f(l)  \, ,
\end{array}
\end{equation}
where
$
f(l):=\sum_{v|l}d(v)\varphi(l/v).
$
%
%Since the functions  $d$ and $\varphi$ are multiplicative so is $f$. Thus to
%compute $f(l)$ we only need to evaluate $f$ at prime powers. Let $p $ be a prime and $k \geq 1$. Then
%$$
%\begin{array}{rcl}
%f(p^k) & = & \displaystyle
%\sum_{v|p^k}d(v)\varphi(l/v)=\sum_{i=0}^kd(p^i)\varphi(p^{k-i})\\[4ex]
%  & = & \displaystyle
%\sum_{i=0}^{k-1}(i+1)(p^{k-i}-p^{k-i-1})+(k+1)\\[4ex]
%  & = & \displaystyle
%p^k+p^{k-1}+\dots+p+1=\frac{p^{k+1}-1}{p-1}\\[3ex]
% & = & \displaystyle p^k\times\frac{p}{p-1}\times(1-p^{k-1})  \ .
%\end{array}
%$$
%Hence,
On exploiting the fact that  $f$ is a multiplicative function, it is readily verified that
$$
f(l)  \, =  \, l\times\prod_{p|l}\frac{p}{p-1}\times\prod_{p|l}(1-p^{-k_p-1}),
$$
where $k_p$ is the largest integer $k$ satisfying $p^k|l$.
 Therefore, by  (\ref{priv}) we have
$$
\textstyle{\frac{6}{\pi^2}} \,l\,\theta(l)\le f(l)\le l\,\theta(l)   \quad \text{where} \quad \theta(l):=\prod_{p|l}\frac{p}{p-1}  \  . $$
Substituting this into (\ref{vb-2}) yields that
\begin{equation}\label{vb-3}
\chi(l)\ge\tfrac12 \sum_{l=1}^h l \, \psi(l)^m\theta(l).
\end{equation}
We will eventually  define $\psi$ to be supported on a subsequence of
$$\textstyle{l_n:=\prod_{i=1}^np_i     \qquad (n \in \N)   } \ , $$ where $p_i$ denotes  the $i$-th  prime.
Obviously,  $\theta(l)$  will then be  strictly increasing on the support of
$\psi$ and furthermore $\lim_{n\to\infty}\theta(l_n)=\infty$.

\vspace*{1ex}

Given an increasing function $F$,
let $\left\{ h_t \right\}_{t \in \N} $ be a subsequence of $\left\{ l_n \right\}_{n \in \N} $  such that for any $T\in\N$
\begin{equation}\label{vb1}
    \frac12\sum_{t=1}^T\theta(h_t)\ge F(8T+8)  \, .
\end{equation}
The existence of such a subsequence is guaranteed by the fact that
$\theta(l_n)  \to \infty$ as $n\to\infty$.
For $t\in\N$,  let $s_t$ denote the number of  terms $l_n$ such that $h_t \le
l_n\le h_{t+1}-1$. Clearly,  $s_t\ge1$ because $\left\{ h_t \right\}$ is a
subsequence of $\left\{ l_n \right\}$. Without loss of generality,  we can assume  that $s_t$ is
increasing since otherwise we work with  an appropriate subsequence of
$\left\{ h_t \right\}$. Now for any natural number $l$ satisfying $h_t \le l\le h_{t+1}-1$, define
$\psi(l)$ by setting
$$
l  \, \psi(l)^m   \, :=  \, \left\{\begin{array}{cl}
                    \displaystyle \frac{1}{s_t} &\quad\text{if } l=l_n\text{ for some
                    }n,\\[2ex]
                    0 & \quad\text{otherwise}.
                  \end{array}
\right.
$$
Set $\psi(l):=0$ for $1 \leq l<h_1$. It is easily seen that $\psi$ is
monotonically decreasing on its support. In view of the definition
of $\psi$,  we have that for every $t\in\N$
\begin{equation}\label{vb-4}
\sum_{l=h_t}^{h_{t+1}-1}l\psi(l)^m=1  \, .
\end{equation}
 Since $\theta(l)$ is
increasing on the support of $\psi$, we have that
\begin{equation}\label{vb-5}
\sum_{l=h_t}^{h_{t+1}-1}l\psi(l)^m\theta(l)\ge \theta(h_t)
\sum_{l=h_t}^{h_{t+1}-1}l\psi(l)^m \stackrel{(\ref{vb-4})}{=}
\theta(h_t)\,.
\end{equation}
Now for any natural number $h \geq h_2$, there exists $T\in\N$ such
that $h_{T+1}\le h<h_{T+2}$ and it follows that
$$
\begin{array}{rcl}
\chi(h) & \stackrel{(\ref{vb-3})}{\ge} & \displaystyle
\sum_{t=1}^{T}\sum_{l=h_t}^{h_{t+1}-1}l\psi(l)^m\theta(l) \
\stackrel{(\ref{vb-5})}{\ge} \ \displaystyle
\sum_{t=1}^{T}\theta(h_t)\stackrel{(\ref{vb1})}{\ge} F(8T+8)\\[4ex]
 & \stackrel{(\ref{vb-4})}{=} & \displaystyle
F\left(8\sum_{t=1}^{T+1}\ \sum_{l=h_t}^{h_{t+1}-1}l\psi(l)^m\right)
 \ \ge \ \displaystyle
F\left(8\sum_{l=1}^{h}l\psi(l)^m\right) \stackrel{(\ref{vb-1})}{\ge}
F(\Phi(h))\,.
\end{array}
$$

\noindent This verifies (\ref{vb2}) and thereby  completes the proof of Lemma \ref{exmp1}.
\end{proof}

\vspace*{2ex}

In view of Theorem \ref{t}, for any  $\psi$ arising  from Lemma \ref{exmp1} we still have that
\begin{equation}\label{asymp}
\lim_{h \to \infty}{\cal N}(\vv X,h)  = \infty    \qquad  {\rm  for \ almost \  all \ } \vv X \, .
\end{equation}
However, Schmidt's theorem fails to describe the asymptotic behavior of ${\cal N}(\vv X,h)$ and therefore the following problem  remains open.

\vspace*{2ex}

\noindent \textbf{Problem.}  For $n=2$ and $\psi$ satisfying the divergent sum condition (\ref{divslv}), describe the asymptotic behavior of ${\cal N}(\vv X,h)$.

\vspace*{1ex}

Lemma \ref{exmp1} can be naturally adapted to the  multivariable setup to show that there is not even a single choice of $n$ and $m$ for which  Schmidt's theorem implies Theorem \ref{t4}.

\begin{lemma}\label{examp2}
Let $n\ge2$ and $F:\Rp\to\Rp$ be an increasing function. Then there
exists an approximating function $\Psi:\Z^n\to\Rp$ satisfying the divergent sum condition of
$(\ref{more1})$ such that  $(\ref{vb2})$ holds.
\end{lemma}

\begin{proof}
Given $F$, let $\psi$ denote the approximating function arising from Lemma~\ref{exmp1}. The lemma  now immediately follows
on  defining $\Psi$ by
$$
\Psi(\vv q):=\left\{ \begin{array}{cl}
                       \psi(|\vv q|) & \text{if }\vv q=(q_1,q_2,0,\dots,0),
                       \\[2ex]
                       0 & \text{otherwise}.
                     \end{array}
\right.
$$
\vspace*{-3ex}
\end{proof}

\vspace*{2ex}

In view of Theorem \ref{t4}, for any  $\Psi$ arising  from Lemma \ref{examp2} we still have (\ref{asymp}). However, Schmidt's theorem is vacuous for such $\Psi$  and describing the asymptotic behavior of ${\cal N}(\vv X,h)$ remains an open problem.

\vspace*{4ex}

\noindent{\it Acknowledgements. }  We would like to thank the referees for their useful comments and  generously  sharing their mathematical insight.  In particular,  the proof of Lemma \ref{nbig1} as it now appears is due to one of the referees  -- our original proof was very much longer.   SV would like to say farewell to Ida Balacki  who told wonderful stories and grew the most amazing basil in London N4.

{\footnotesize

\def\cprime{$'$}

}

\vspace{2ex}

{\small

\noindent Victor V. Beresnevich: Department of Mathematics,
University of York,

\vspace{-2mm}

\noindent\phantom{Victor V. Beresnevich: }Heslington, York, YO10
5DD, England.

\vspace{-2mm}

\noindent\phantom{Victor V. Beresnevich: }e-mail: vb8@york.ac.uk

\vspace{5mm}

\noindent Sanju L. Velani: Department of Mathematics, University
of York,

\vspace{-2mm}

\noindent\phantom{Sanju L. Velani: }Heslington, York, YO10 5DD,
England.

\vspace{-2mm}

\noindent\phantom{Sanju L. Velani: }e-mail: slv3@york.ac.uk

}

\end{document}